\documentclass{amsart}
\usepackage{amssymb,stmaryrd}
\usepackage[pdfusetitle,colorlinks,citecolor=blue,urlcolor=black,linkcolor=black]{hyperref}

\newtheorem{thm}{Theorem}[section]
\newtheorem{fact}[thm]{Fact}
\newtheorem{prop}[thm]{Proposition}
\newtheorem{cor}[thm]{Corollary}
\newtheorem{claim}{Claim}[thm]

\theoremstyle{definition}
\newtheorem{defn}[thm]{Definition}
\newtheorem{q}{Problem}

\DeclareMathOperator{\U}{U}
\DeclareMathOperator{\cf}{cf}

\DeclareMathOperator{\reg}{Reg}
\DeclareMathOperator{\ads}{ADS}
\DeclareMathOperator{\refl}{Refl}

\DeclareMathOperator{\nacc}{nacc}
\DeclareMathOperator{\acc}{acc}
\DeclareMathOperator{\tp}{TP}

\DeclareMathOperator{\otp}{otp}
\DeclareMathOperator{\fs}{FS}

\newcommand\s{\subseteq}
\newcommand\sq{\sqsubseteq}

\newcommand\stick{{{\ensuremath \mspace{2mu}\mid\mspace{-12mu} {\raise0.6em\hbox{$\bullet$}}}}}
\newcommand*\axiomfont[1]{\textsf{\textup{#1}}}
\newcommand\zfc{\axiomfont{ZFC}}
\newcommand\gch{\axiomfont{GCH}}
\newcommand\sch{\axiomfont{SCH}}

\renewcommand{\mid}{\mathrel{|}\allowbreak}
\renewcommand{\restriction}{\mathbin\upharpoonright}

\author {Assaf Rinot}
\address{Department of Mathematics, Bar-Ilan University, Ramat-Gan 5290002, Israel.}
\urladdr{http://www.assafrinot.com}
\title[Jónsson successor of a singular]{May the successor of a\\singular cardinal be Jónsson?}

\begin{document}
\begin{abstract} In this survey, we collect necessary conditions for the successor of a singular cardinal to be J{\'o}nsson.
\end{abstract}
\date{A preprint as of 13-Dec-2023. For the latest version, visit \textsf{http://p.assafrinot.com/s02}.}

\maketitle

\section{Introduction}

Our starting point is Ramsey's theorem \cite{onaproblemofformallogic}
asserting that for every 2-coloring of the unordered pairs of an infinite set $X$,
$c:[X]^2 \rightarrow 2$, there exists an infinite subset $Y\s X$
that is \emph{$c$-homogeneous} meaning that $c$ is constant over $[Y]^2$.
Soon after, Sierpi{\'n}ski \cite{MR1556708} demonstrated that Ramsey's theorem does not generalize to the uncountable
by constructing a coloring $c:[\mathbb R]^2\rightarrow2$ satisfying that every uncountable $Y\s\mathbb R$
is \emph{$c$-omnichromatic}, meaning that $c\restriction [Y]^2$ is a surjection.

Instead of interpreting Sierpi{\'n}ski's result as a disappointment, set theorists
view witnesses to the failure of generalized Ramsey theorem as \emph{strong colorings} and set their goal to make them as strong as possible.
A succinct way to state results of this form is using the following \emph{arrow notation}:
\begin{defn}[{\cite[\S18]{MR202613}}] \label{hungarian} For cardinals $\kappa,\theta,\lambda$ and a positive integer $n$:
\begin{itemize}
\item $\kappa\nrightarrow[\lambda]^n_\theta$ asserts the existence of a coloring $c:[\kappa]^2\rightarrow\theta$ such that $c\restriction [Y]^n$ is surjective for every $Y\s\kappa$ of size $\lambda$;
\item $\kappa\nrightarrow[\lambda]^{<\omega}_\theta$ asserts the existence of a coloring $c:[\kappa]^{<\omega}\rightarrow\theta$ such that $c\restriction [Y]^{<\omega}$ is surjective for every $Y\s\kappa$ of size $\lambda$.
\end{itemize}
\end{defn}

An infinite cardinal $\kappa$ is \emph{J{\'o}nsson} iff $\kappa\rightarrow[\kappa]^{<\omega}_\kappa$ holds, i.e., $\kappa\nrightarrow[\kappa]^{<\omega}_\kappa$ fails.
Rowbottom \cite{MR323572} proved that every measurable cardinal is a (regular) J{\'o}nsson cardinal,
and Prikry \cite{MR262075} proved that a measurable cardinal may be turned into a singular J{\'o}nsson cardinal via a cardinal-preserving forcing.
Tryba \cite{MR788256} and independently Woodin proved that if $\kappa$ is a J{\'o}nsson cardinal then every stationary subset of $\kappa$ reflects.
This was then improved by Todor{\v{c}}evi{\'c} \cite{TodActa} who derived the same conclusion from $\kappa\rightarrow[\kappa]^2_\kappa$.
In particular, the successor of a regular cardinal is never J{\'o}nsson.
This survey centers around the open problem of whether the successor of a singular cardinal may consistently be J{\'o}nsson.

\subsection{Notation and conventions}\label{conventions}
Throughout this paper, $\kappa$ denotes a regular uncountable cardinal,
$\lambda$ and $\mu$ denote infinite cardinals, and $\theta$ a (possibly finite) cardinal.
Let $E^\kappa_\mu:=\{\alpha < \kappa \mid \cf(\alpha) = \mu\}$,
and define $E^\kappa_{\le \mu}$, $E^\kappa_{<\mu}$, $E^\kappa_{\ge \mu}$, $E^\kappa_{>\mu}$,  $E^\kappa_{\neq\mu}$ analogously.
$\reg(\kappa)$ stands for the set of all infinite regular cardinals below $\kappa$.
We identify $[\kappa]^2$ with $\{ (\alpha,\beta)\mid \alpha<\beta<\kappa\}$.
For a set of ordinals $A$, we write $\acc(A) := \{\alpha\in A \mid \sup(A \cap \alpha) = \alpha > 0\}$
and $\nacc(A) := A \setminus \acc(A)$.

\section{Cardinal arithmetic}

\begin{thm}[{Erd{\H{o}}s-Hajnal-Rado, \cite{MR202613}}]\label{thm21} Suppose that $\lambda$ is an infinite cardinal satisfying that $2^\lambda=\lambda^+$.
Then $\lambda^+\nrightarrow[\lambda^+]^2_{\lambda^+}$ holds.
\end{thm}
\begin{proof} Using $2^\lambda=\lambda^+$, let $\langle a_\alpha\mid \alpha<\lambda^+\rangle$ enumerate all bounded subsets of $\lambda^+$.
We shall need the following standard claim concerning disjoint refinements.
\begin{claim} Let $\beta<\lambda^+$. There exists a pairwise disjoint subfamily $\mathcal B_\beta\s[\beta]^\lambda$
satisfying that for every $\alpha<\beta$ such that $a_\alpha\in[\beta]^\lambda$, there is $b\in\mathcal B_\beta$ with $b\s a_\alpha$.
\end{claim}
\begin{proof} To avoid trivialities, suppose that $A_\beta:=\{\alpha<\beta\mid a_\alpha\in[\beta]^\lambda\}$ is nonempty,
so that $0<|A_\beta|\le\lambda$.
Fix a bijection $\pi:\lambda\leftrightarrow\lambda\times\mathcal A_\beta$ and then recursively construct an injection $f:\lambda\rightarrow\beta$ satisfying that for every $i<\lambda$:
$$(\pi(i)=(j,\alpha))\implies(f(i)\in a_\alpha).$$
For every $\alpha\in A_\beta$, the set $b_\alpha:=\{ f(i)\mid \exists j\,(\pi(i)=(j,\alpha))\}$
is a $\lambda$-sized subset of $a_\alpha$, and since $f$ is injective, $\mathcal B_\beta:=\{ b_\alpha\mid \alpha\in A_\beta\}$ is a pairwise disjoint family.
\end{proof}

Let $\langle \mathcal B_\beta\mid\beta<\lambda^+\rangle$ be given by the preceding claim.
Pick a coloring $c:[\lambda^+]^2\rightarrow\lambda^+$ satisfying that for every $\beta<\lambda^+$ for every $b\in\mathcal B_\beta$, $c[b\times\{\beta\}]=\beta$.
To see this works, let $Y$ be a subset of $\lambda^+$ of full size, and let $\tau<\lambda^+$ be any prescribed color,
and we shall show that $\tau\in c``[Y]^2$.

As $|Y|=\lambda^+$, we may find a large enough $\epsilon<\lambda^+$ such that $|Y\cap\epsilon|=\lambda$ and then find an $\alpha<\lambda^+$ such that $Y\cap\epsilon=a_\alpha$.
Pick $\beta\in Y$ above $\max\{\epsilon,\alpha,\tau\}$.
Evidently, $a_\alpha\in[\beta]^\lambda$, so we may pick $b\in\mathcal B_\beta$ such that $b\s a_\alpha$.
Altogether, $b\s Y$ and $\tau\in\beta=c[b\times\{\beta\}]\s c``[Y]^2$.
\end{proof}

A few remarks are in order here:
\begin{enumerate}
\item The proof of Theorem~\ref{thm21} is a template of a proof of anti-Ramsey consequences of the $\gch$. See, e.g., \cite[Theorem~9]{MR0505558} and \cite[Theorem~2]{MR4143159}
for elaborations of this argument.
\item It is not hard to see that Theorem~\ref{thm21} (and its successors) remain valid after weakening the hypothesis of $2^\lambda=\lambda^+$ to the strick principle $\stick(\lambda^+)$.
In the special case of $\lambda$ regular, stronger consequences may be derived --- see \cite[Theorem~6.3]{paper45}.
\item In \cite[pp.~290--291]{TodActa}, Todor{\v{c}}evi{\'c}   shows that the special case of $\lambda=\aleph_0$ of Theorem~\ref{thm21}
follows from a classical theorem of Sierpi\'nski \cite{sierpinski1934hypothese}
asserting that the continuum hypothesis gives rise to a sequence $\langle f_n\mid n<\omega\rangle$ of functions from $\aleph_1$ to $\aleph_1$
such that for every uncountable $Y\s\aleph_1$, for all but finitely many $n<\omega$, $f_n[Y]=\aleph_1$.
This finding was extended in \cite{paper55}.
\item A simple stretching argument shows that $\lambda^+\nrightarrow[\lambda^+]^2_\theta$ with $\theta=\lambda^+$
follows from $\lambda^+\nrightarrow[\lambda^+]^2_\theta$ with merely $\theta=\lambda$.
This was extended in \cite[\S3]{paper50}.
\end{enumerate}

The upcoming Fact~\ref{fact23} reduces the arithmetic hypothesis of Theorem~\ref{thm21} to a postulation coming from {\it pcf} theory, as follows.
First, given an infinite cardinal $\lambda$ and a subset $\Theta\s\reg(\lambda)$,
define an ordering $<^*$ over $\prod\Theta$ by letting $f<^*g$ iff $\sup\{ \theta\in\Theta\mid f(\theta)\ge g(\theta)\}<\sup(\Theta)$.
Then, let
\begin{itemize}
\item $\mathfrak b(\Theta)$ denote the least size of unbounded family in $(\prod\Theta,<^*)$, and let
\item $\mathfrak d(\Theta)$ denote the least size of a cofinal family in $(\prod\Theta,<^*)$.
\end{itemize}

\begin{fact}[Shelah, Theorem~1.5 and Claim~2.1 of \cite{Sh:355}]\label{fact22} Suppose that $\lambda$ is a singular cardinal. Then:
\begin{enumerate}
\item There exists a cofinal subset $\Theta\s\reg(\lambda)$ such that $\mathfrak d(\Theta)=\lambda^+$ (and hence also $\mathfrak b(\Theta)=\lambda^+$);
\item If $\cf(\lambda)>\omega$, then the above $\Theta$ may be taken to be $\{\mu^+\mid \mu\in C\}$
for some sparse enough club $C$ in $\lambda$.
\end{enumerate}
\end{fact}

\begin{fact}[Todor\v{c}evi\'c, {\cite[pp.~289]{TodActa}; see \cite[Theorem~4.11]{MR1086455} for a proof}]\label{fact23} Suppose that $\lambda$ is a singular cardinal.
If there exists a cofinal subset $\Theta\s\reg(\lambda)$ such that $\mathfrak d(\Theta)=\lambda^+$
and $\theta\nrightarrow[\theta]^2_\theta$ holds for every $\theta\in\Theta$, then
$\lambda^+\nrightarrow[\lambda^+]^2_{\lambda^+}$ holds.\footnote{An analogous pump-up fact holds true for J{\'o}nsson-ness. See \cite[Conclusion~4.6($\gamma$)]{Sh:355}.}
\end{fact}

Recall that by the main result of \cite{TodActa}, $\theta\nrightarrow[\theta]^2_\theta$ holds whenever $\theta$ is a successor of a regular cardinal,
thus we arrive at the following informative conclusion.
\begin{cor}\label{cor24}
\begin{enumerate}
\item If $\lambda$ is a singular cardinal satisfying $\lambda^+\rightarrow[\lambda^+]^2_{\lambda^+}$,
then $\lambda$ is the limit of (weakly) inaccessibles;
\item If $\lambda$ is the first cardinal to satisfy $\lambda^+\rightarrow[\lambda^+]^2_{\lambda^+}$, then $\lambda$ is a singular cardinal of countable cofinality.
\end{enumerate}
\end{cor}

Question 6.7 of the recent preprint \cite{2211.13361} is equivalent to asking whether  hypothesis of $\mathfrak{d}(\Theta)$ = $\lambda^+$
of Fact~\ref{fact23} may be reduced to $\mathfrak{b}(\Theta)$ = $\lambda^+$. We answer this question in the affirmative.

\begin{prop}\label{prop25} Suppose that $\lambda$ is a singular cardinal.
If there exists a cofinal subset $\Theta\s\reg(\lambda)$ such that $\mathfrak b(\Theta)=\lambda^+$
and $\theta\nrightarrow[\theta]^2_\theta$ holds for every $\theta\in\Theta$, then
$\lambda^+\nrightarrow[\lambda^+]^2_{\lambda^+}$ holds.
\end{prop}
\begin{proof} Suppose that $\Theta$ is a cofinal subset of $\reg(\lambda)$ such that $\mathfrak b(\Theta)=\lambda^+$
and $\theta\nrightarrow[\theta]^2_\theta$ holds for every $\theta\in\Theta$.
By possibly passing to a cofinal subset, we may assume that $\otp(\Theta)=\cf(\lambda)$,
and let $\langle \lambda_i\mid i<\cf(\lambda)\rangle$ be the increasing enumeration of $\Theta$.
For each $i<\cf(\lambda)$, let $c_i:[\lambda_i]^2 \rightarrow \lambda_i$ be a witness for $\lambda_i \nrightarrow [\lambda_i]^2_{\lambda_i}$.

Using $\mathfrak{b}(\Theta)=\lambda^+$, we may fix an unbounded sequence $\vec f=\langle f_\alpha\mid \alpha<\lambda^+\rangle$ of functions in $\prod\Theta$
satisfying, in addition, that $f_\alpha<^* f_\beta$ for all $\alpha<\beta<\lambda^+$.
This buys us the following standard feature.

\begin{claim} For every subset $Z\s\lambda^+$ of full size, the following set
$$I_Z:=\{ i<\cf(\lambda)\mid \sup\{ \epsilon<\lambda_i\mid \sup\{\beta\in Z\mid f_\beta(i)=\epsilon\}=\lambda^+\}=\lambda_i\}$$
is cofinal in $\cf(\lambda)$.
\end{claim}
\begin{proof} Let $Z\s\lambda^+$ of full size be given.
Fix a function $g\in\prod_{i<\cf(\lambda)}\lambda_i$ satisfying that for every $i\in\cf(\lambda)\setminus I_Z$:
$$g(i):=\sup\{ \epsilon<\lambda_i\mid \sup\{\beta\in Z\mid f_\beta(i)=\epsilon\}=\lambda^+\}+1.$$
As $\vec f$ is unbounded, there exists some $\alpha\in\lambda^+$ such that $\{ i<\cf(\lambda)\mid f_\alpha(i)\ge g(i)\}$ is cofinal in $\cf(\lambda)$.
Recalling that $f_\alpha<^* f_\beta$ for all $\alpha<\beta<\lambda^+$,
we infer the existence of a final segment $I$ of $\{ i<\cf(\lambda)\mid f_\alpha(i)\ge g(i)\}$ such that, for cofinally many $\beta\in Z$,
$$I\s \{ i<\cf(\lambda)\mid f_\beta(i)\ge g(i)\}.$$
By the pigeonhole principle, $I\s I_Z$. In particular, $I_Z$ is cofinal in $\cf(\lambda)$.
\end{proof}

Next, by \cite[Theorem~1]{MR3087058}, we may fix a map $d:[\lambda^+]^2 \rightarrow [\lambda^+]^2\times \cf(\lambda)$
satisfying that for every $Y \s \lambda^+$ of full size, there is a stationary $Z \s \lambda^+$ such that $d``[Y]^2$ covers $[Z]^2\times \cf(\lambda)$.
Recall that by a stretching argument, it suffices to prove that $\lambda^+\nrightarrow[\lambda^+]^2_\lambda$ holds.
To this end, we define a coloring $c:[\lambda^+]^2 \rightarrow \lambda$ by letting $$c(\alpha,\beta):=c_i (f_\gamma(i),f_\delta(i))$$ whenever $d(\alpha,\beta)=(\gamma,\delta,i)$.

To see this works, suppose that we are given $Y\s\lambda^+$ of full size,
and a prescribed color $\tau<\lambda$.
Pick $Z \s \lambda^+$ of full size such that $d``[Y]^2$ covers $[Z]^2\times \cf(\lambda)$.
Pick a large enough $i\in I_Z$ such that $\tau<\lambda_i$.
Since $E_i:=\{ \epsilon<\lambda_i\mid \sup\{\beta\in Z\mid f_\beta(i)=\epsilon\}=\lambda^+\}$ is cofinal in $\lambda_i$,
we may find a pair $\epsilon<\epsilon'$ of ordinals in $E_i$ such that $c_i(\epsilon,\epsilon')=\tau$.
Pick $\gamma\in Z$ such that $f_\gamma(i)=\epsilon$ and then pick $\delta\in Z$ above $\gamma$ such that $f_\delta(i)=\epsilon'$.
Let $\alpha<\beta$ be a pair of ordinals in $Y$ such that $d(\alpha,\beta)=(\gamma,\delta,i)$.
Then
$$c(\alpha,\beta)=c_i(f_\gamma(i),f_\delta(i))=c_i(\epsilon,\epsilon')=\tau,$$
as sought.
\end{proof}

Recall that the singular cardinals hypothesis ($\sch$) is the assertion that $\sch_\lambda$ holds for every singular cardinal $\lambda$,
where $\sch_\lambda$ means that if $2^{\cf(\lambda)}<\lambda$, then $2^\lambda=\lambda^+$.
Motivated by Proposition~\ref{prop25},
we formulate the following particular failure of the $\sch$:
\begin{defn} For a property $\varphi$ of a cardinal (such as being inaccessible) and a singular cardinal $\lambda$,
the \emph{$\varphi$ failure of $\sch_\lambda$} asserts that
for every cofinal $\Theta\s\reg(\lambda)$ with $\mathfrak{b}(\Theta)=\lambda^+$,
it is the case that co-boundedly many $\theta\in\Theta$ has property $\varphi$.
\end{defn}

\begin{q}\label{p1} Is the weakly compact failure of $\sch_\lambda$ consistent (from a large cardinals hypothesis)?
\end{q}

Note that by \cite[Theorem~1]{MR0379200}, the strongly compact failure of $\sch_\lambda$ is inconsistent.
Adolf informed us that the consistency of the weakly compact failure of $\sch_\lambda$ requires a Woodin cardinal.

Ben-Neria suggested that an affirmative answer to Problem~\ref{p1} may be obtained using the forcing of Merimovich from \cite{MR2805299}.
Specifically, suppose $j : V \to M$ is a $V$-definable nontrivial elementary embedding
with critical point $\lambda$. Denote $\mu := j(\lambda)$ and $\kappa := j(\mu)$ and further assume ${}^{<\mu} M \subseteq M$ and $V_{\kappa} \subseteq M$.
Let $\vec E$ be the extender derived from $j$ with measures $\vec E(d)$ for $d \in [\kappa]^{<\mu}$,
and finally let $\mathbb{P}_{\vec E}$ be the associated extender-based Prikry forcing of \cite{MR2805299}.
Merimovich proved that this forcing changes the cofinality of $\lambda$ to $\omega$,  does not add bounded subsets of $\lambda$, collapses the cardinals in the interval $(\lambda,\mu)$ and only those, and adds $\kappa$-many $\omega$-sequences to $\lambda$. Ben-Neria claims that $V^{\mathbb{P}_{\vec E}}$ witnesses the measurable failure of $\sch_\lambda$.

\section{Compactness and incompactness}
We mentioned in the introduction that
the existence of a nonreflecting stationary subset of $\kappa$ implies that $\kappa\nrightarrow[\kappa]^2_\kappa$ holds. This was extended in \cite{paper15}.
In the context of successors of singulars, a stronger consequence may be derived from the failure of $\kappa\nrightarrow[\kappa]^2_\kappa$, as follows.
\begin{fact}[Eisworth, \cite{MR3051629}]\label{fact31} Suppose that $\lambda$ is a singular cardinal such that $\lambda^+\rightarrow[\lambda^+]^2_{\lambda^+}$ holds.
Then every family $\mathcal S$ of less than $\cf(\lambda)$-many stationary subsets of $\lambda^+$
reflects simultaneously, that is,
there exists an ordinal $\delta<\lambda^+$ of uncountable cofinality such that $S\cap\delta$ is stationary in $\delta$
for every $S\in\mathcal S$.
\end{fact}

The combination of the preceding fact with the results of the previous section prompts the study of models in which there exists a singular cardinal $\lambda$ of countable cofinality
such that $\sch_\lambda$ fails and simultaneous reflection of stationary subsets of $\lambda^+$ holds.
The existence of such a model was demonstrated only very recently.
In \cite{paper42}, Poveda, Rinot and Sinapova obtained such a model using {iterated Prikry-type forcing}, continuing the work of Sharon \cite{Sharon}.
In \cite{bhu}, Ben-Neria, Hayut and Unger obtained such a model using {iterated ultrapowers}. A third model was constructed by Gitik in \cite{GitikRef}.

A possible explanation of the difficulty in getting such a model is in the fact that the failure of $\sch$
implies the failure of reflection of two-cardinal stationary sets. Specifically,
the failure of $\sch_\lambda$ implies the existence of a so-called \emph{better scale} at $\lambda$
which the proof of \cite[Theorem~4.1]{cfm} shows to imply the combinatorial principle $\ads_\lambda$.
Finally, by \cite[Theorem~4.2]{cfm}, if $\ads_\lambda$ holds for a singular cardinal $\lambda$ of countable cofinality,
then the reflection principle $\refl^*([\lambda^+]^{\aleph_0})$ fails.
\begin{q}
Find combinatorial consequences of the weakly compact failure of $\sch_\lambda$.
\end{q}

We continue with two more facts connecting incompactness to strong colorings.

\begin{fact}[{Jensen, \cite[Lemma~6.6]{MR309729}; Shore \cite[Lemma~1]{MR0371662}}]\label{fact32}\

If there is a $\kappa$-Souslin tree, then $\kappa \nrightarrow [\kappa]^2_\kappa$ holds.
\end{fact}

\begin{fact}[\cite{paper18}]
If  $\square(\kappa)$  holds,\footnote{The conclusion follows from weaker hypotheses such as $\square(\kappa,{<}\omega,\sq_\nu)$ whose definition may be found in \cite[Definition~1.16]{paper29}.} then so does $\kappa \nrightarrow [\kappa]^2_\kappa$.
\end{fact}

The hypotheses of the last two facts have to do with particular forms of $\kappa$-Aronszajn trees. So, we ask:
\begin{q}\label{p3} Suppose that $\lambda$ is a singular cardinal and there exists a $\lambda^+$-Aronszajn tree.
Does $\lambda^+ \nrightarrow [\lambda^+]^2_\lambda$ hold?
\end{q}

Taking into account all of the results discussed so far, we arrive at the following difficult problem:
\begin{q} Modulo a large cardinals hypothesis, is the conjunction of the following consistent for some singular cardinal $\lambda$?
\begin{itemize}
\item Weakly compact failure of $\sch_\lambda$;
\item $\tp(\lambda^+)$, i.e., there are no $\lambda^+$-Aronszanjn trees;
\item Every finite family of stationary subsets of $\lambda^+$ reflect simultaneously.
\end{itemize}
\end{q}

The first model to satisfy the (usual) failure of $\sch_\lambda$ together with $\tp(\lambda^+)$ holding was constructed by Neeman in \cite{MR2665784}.
More recent works in this vein include \cite{MR4127897}.

\medskip

We end this section by pointing out that the combination of Fact~\ref{fact32}  and Problem~\ref{p3}
gives rise to the following problem (a variation of \cite[Problem~1]{MR2162107} that we studied in \cite{paper32}):
\begin{q}\label{p5}
Suppose that $\lambda$ is a singular cardinal and there exists a $\lambda^+$-Aronszajn tree.
Does there exist a $\lambda^+$-Souslin tree?
\end{q}

We find the preceding problem to be of  interest even in the context of $\gch$.
The point is that the main result of \cite{paper26}
shows that  in this context, singularizations of a regular cardinal $\lambda$ tend to do introduce $\lambda^+$-Souslin trees.
Finally, in view of the fact that the existence of a $\lambda^+$-Aronszajn tree is equivalent to $\square(\lambda^+,\lambda)$,
the following variation of Problem~\ref{p5} emerges:\footnote{Furthermore, the implication of Problem~\ref{p6} for $\lambda$ regular and uncountable holds true \cite{paper37}.}
\begin{q}\label{p6}
Suppose that $\lambda$ is a singular cardinal and $\square(\lambda^+,{<}\lambda)$ and $\gch$ both hold.
Does there exist a $\lambda^+$-Souslin tree?
\end{q}

In \cite{paper24}, an affirmative answer is given under the stronger hypothesis of $\square(\lambda^+)$.
By \cite[Corollary~2.24]{paper29}, the hypothesis of Problem~\ref{p6} is sufficient for the construction of a \emph{distributive} $\lambda^+$-Aronszajn tree.

\medskip

\section{Reductions and approximations}

We mentioned earlier that $\lambda^+\nrightarrow[\lambda^+]^2_\theta$ with $\theta=\lambda^+$
follows from $\lambda^+\nrightarrow[\lambda^+]^2_\theta$ with merely $\theta=\lambda$.
We now recall a further reduction.

\begin{fact}[Eisworth, \cite{MR3087058}] Suppose that $\lambda$ is a singular cardinal.
If $\lambda^+ \nrightarrow [\lambda^+]^2_\theta$ holds for arbitrarily large $\theta<\lambda$,
then $\lambda^+ \nrightarrow [\lambda^+]^2_{\lambda}$ holds.
\end{fact}

\begin{fact}[{Shelah, \cite[Conclusion 4.1]{Sh:355}}]\label{fact42}  \

For every singular cardinal $\lambda$,
$\lambda^+ \nrightarrow [\lambda^+]^2_{\cf(\lambda)}$  holds.
\end{fact}
In view of the last two facts, we ask:
\begin{q} Suppose that $\lambda$ is a singular cardinal.   Does $\lambda^+ \nrightarrow [\lambda^+]^2_{\cf(\lambda)^+}$  hold?
\end{q}
Note that $\lambda^+ \nrightarrow [\lambda^+]^2_{\cf(\lambda)^+}$  is equivalent to the syntactically-weaker principle $\U(\lambda^+,2,\cf(\lambda)^+,2)$ of \cite[Definition~1.2]{paper34}.
We do not know whether this equivalency remains true replacing $\cf(\lambda)^+$ by an arbitrary regular cardinal $\theta<\lambda$.

\begin{defn}
Given a coloring  $c:[\lambda^+]^2 \rightarrow \theta$ and a cardinal $\mu\le\lambda$, consider the following notion of forcing $$\mathbb{P}_{c,\mu}:=(\{ x \in [\lambda^+]^{<\mu}\mid c\restriction [x]^2\text{ is constant }\},\supseteq)$$
for adding a large $c$-homogeneous set, so that $c$ will not witness $\lambda^+ \nrightarrow[\lambda^+]^2_\theta$.
\end{defn}
\begin{fact}[{\cite[Theorem~1]{paper13}}]
Suppose that $\lambda$ is a singular cardinal and $\theta\le\lambda^+$ is any cardinal.
If $\lambda^+ \nrightarrow [\lambda^+]^2_\theta$ holds, then it may be witnessed by a coloring $c:[\lambda^+]^2 \rightarrow \theta$
for which $\mathbb{P}_{c,\mu}$  has the $\lambda^+$-cc for every cardinal $\mu$ such that $\lambda^{<\mu}=\lambda$.
\end{fact}

This ensures that cardinals above $\lambda$ will not be collapsed. But what about the ones from below? Here, we have some bad news
indicating that $\mathbb{P}_{c,\mu}$ is not the right poset for this task.
\begin{fact}[{\cite[Proposition 2.10 and 2.12]{paper45}}]
Suppose that $c:[\lambda^+]^2 \rightarrow 2$ is a coloring  for some singular cardinal $\lambda$. Then:
\begin{itemize}
\item  $\mathbb{P}_{c,\lambda}$ has an antichain of size $\lambda^+$ consisting of pairwise disjoint sets;

\item If $\lambda$ is the limit of strongly compacts, then the above is true already for $\mathbb{P}_{c,\cf(\lambda)^+}$.
\end{itemize}
\end{fact}

\begin{q}\label{p8}
Given a coloring  $c:[\lambda^+]^2 \rightarrow \theta$ witnessing $\lambda^+ \nrightarrow[\lambda^+]^2_\theta$,
is there a cofinality-preserving notion of forcing for killing $c$?
Identify additional features of $c$ that would enable an affirmative answer.
\end{q}

To demonstrate the difficulty in solving Problem~\ref{p8}, examine the very proof of Fact~\ref{fact42}
and see what it takes to kill the particular coloring constructed there.

\medskip

The next reduction tells us that we may focus our attention on sets thicker than just cofinal.

\begin{fact}[{\cite[Theorem~2]{paper13}}]  Suppose that $\lambda$ is a singular cardinal.
If there are a cardinal $\mu<\lambda$ and a coloring $c:[\lambda^+]^2 \rightarrow \theta$ such that  $c``[S]^2=\theta$  for every \emph{stationary}  $S \s E^{\lambda^+}_{>\mu}$,
then $\lambda^+ \nrightarrow [\lambda^+]^2_\theta$ holds.
\end{fact}

\begin{q}  Identify interesting ideals $J$ over $\lambda^+$ for which
$\zfc$ proves the existence of a coloring $c:[\lambda^+]^2 \rightarrow \lambda$ satisfying  $c``[B]^2=\lambda$ for every $B\in J^+$.
\end{q}

The next reduction tells us that instead of implementing prescribed colors as $c(\alpha,\beta)$ for some pair $\alpha<\beta$ of ordinals from a given large set $Y$,
it suffices to do so as $c(i,\beta)$ where only $\beta$ is required to come from $Y$.

\begin{fact}[{\cite[Lemma~8.9(2)]{paper47}}]
Suppose a singular cardinal $\lambda$ is a strong limit or satisfies $\aleph_\lambda>\lambda$.
If there exists a coloring $c:\lambda\times \lambda^+ \rightarrow \lambda$ such that for every $Y \s \lambda^+$ of full size,
there is $i<\lambda$ with $c[\{i\}\times Y]=\lambda$, then $\lambda^+ \nrightarrow [\lambda^+]^2_\lambda$ holds.
\end{fact}

Unfortunately, in \cite{paper53} it was shown that in the model of \cite{GaSh:949}
there is no such coloring $c:\lambda\times\lambda^+\rightarrow\lambda$.
But we do get two approximations of such a coloring in $\zfc$:

\begin{fact}[{\cite[\S6]{paper53}}] Suppose that $\lambda$ is a singular cardinal.
\begin{enumerate}
\item  There is a coloring $c:\lambda\times \lambda^+ \rightarrow \lambda$ such that for every $Y \s \lambda^+$ of full size,
there is $i<\lambda$ with  $\otp(c[\{i\}\times Y])=\lambda$;
\item For every cardinal $\theta<\lambda$,
there is a coloring $c:\lambda\times \lambda^+ \rightarrow \theta$  such that for every $Y \s \lambda^+$ of full size,
there is $i<\lambda$ with $c[\{i\}\times Y]=\theta$.
\end{enumerate}
\end{fact}

To appreciate the fact that Clause~(1) holds in $\zfc$ for $\lambda$ singular, note that the same assertion for $\lambda$ regular is equivalent to $\mathfrak{b}_\lambda=\lambda^+$ (see \cite[Lemma~6.1]{paper53}).

\section{Club guessing}
For a set of ordinals $S$, a \emph{$C$-sequence over $S$}
is a sequence $\vec C=\langle C_\delta\mid \delta\in S\rangle$ satisfying that
for every $\delta\in S$, $C_\delta$ is a closed subset of $\delta$ with $\sup(C_\delta)=\sup(\delta)$.
\begin{defn}[Shelah] A $C$-sequence $\vec C=\langle C_\delta\mid \delta\in S\rangle$
over a stationary subset $S\s\kappa$ is said to \emph{guess clubs}
iff for every club $D\s\kappa$, there is some $\delta\in S$ with $C_\delta \s D$.
\end{defn}

By \cite[\S1]{Sh:365}, for every stationary $S\s\kappa$ such that $\sup\{ \cf(\delta)^+\mid \delta\in S\}<\kappa$,
there exists a $C$-sequence over $S$ that guesses clubs.

\begin{defn} A $C$-sequence $\vec C=\langle C_\delta\mid \delta\in S\rangle$
is \emph{uninhibited} iff for club many $\delta\in S$, for every $\mu\in\reg(\delta)$, $\sup(\nacc(C_\delta)\cap E^{\delta}_{\ge\mu})=\delta$.
\end{defn}

\begin{defn}\label{defn53} Given a $C$-sequence $\vec C=\langle C_\delta\mid \delta\in S\rangle$ over some stationary $S\s\kappa$,
consider the following corresponding ideal:
$$J(\vec C):=\{ A \s \kappa\mid \exists\text{club }D \s \kappa\forall \delta\in S\exists\mu\in\reg(\delta)[\sup(\nacc(C_\delta)\cap E^{\delta}_{\ge\mu}\cap D\cap A)<\delta]\}.$$
\end{defn}

By \cite[Claim~2.4]{Sh:365} and \cite[Theorem~2]{EiSh:819}, for every cardinal $\lambda$ of uncountable cofinality,
there exists an uninhibited club-guessing $C$-sequence $\vec C$ over $E^{\lambda^+}_{\cf(\lambda)}$. In particular,
in this case, the corresponding ideal $J(\vec C)$ is a proper ideal.
The following is Question~2.4 of \cite{EiSh:819} and is still open:

\begin{q}\label{p10} Suppose that $\lambda$ is a singular cardinal of countable cofinality.
Must there exist an uninhibited club-guessing $C$-sequence over $E^{\lambda^+}_{\cf(\lambda)}$?
\end{q}

A pump-up result (proved using a simple stretching argument) asserts that if $\lambda$ is not a J{\'o}nsson cardinal, then neither is $\lambda^+$.
In \cite[Lemma~1.9]{Sh:365}, Shelah proved a deep pump-up theorem asserting that  for every singular cardinal $\lambda$,
if there exists a $C$-sequence $\vec C$ over $E^{\lambda^+}_{\cf(\lambda)}$ for which
there exists a $J(\vec C)$-positive subset of $\{ \beta<\lambda^+\mid \cf(\beta) \text{ is not J{\'o}nsson}\}$,
then $\lambda^+$ is not J{\'o}nsson.
\begin{q}
Is $\lambda^+ \rightarrow[\lambda^+]^2_\lambda$ equivalent to the Jónsson-ness of $\lambda^+$?  to  $\lambda^+ \rightarrow[\lambda^+]^n_\lambda$ for some positive integer $n$?
\end{q}

J{\'o}nsson cardinals are known to have tight connections to problems in algebra.
By \cite[Corollary~2.8]{paper27},
$\kappa$ is not  J{\'o}nsson iff
the following strong failure of the higher analog of Hindman's theorem holds true:
for every Abelian group $G$ of size $\kappa$,
there exists a coloring $c:G\rightarrow\kappa$ such that for every $Y\s G$ of full size,
$c\restriction\fs(Y)$ is surjective, where
$\fs(Y)$ stands for the set of all finite sums $y_1+\cdots+y_n$ of distinct elements of $Y$.
A particularly nice witness to a cardinal $\kappa$ not being J{\'o}nsson
is the existence of a \emph{Shelah group} of size $\kappa$, that is, a group $G$ (of size $\kappa$) for which there exists a positive integer $n$
such that for every $Y\s G$  of full size, every element of $G$ may be
written as a group word of length $n$ in the elements of $Y$.
In \cite{Sh:69}, Shelah proved that $2^\lambda=\lambda^+$ entails the existence of a Shelah group of size $\lambda^+$.
In \cite{paper60}, it was shown that the arithmetic hypothesis is redundant in the case that $\lambda$ is regular. Thus, we ask:
\begin{q}  Suppose that $\lambda$ is a singular cardinal such that $\lambda^+$ is not Jónsson. Does there exist a Shelah group of size $\lambda^+$?
\end{q}

Coming back to the context of coloring pairs, we have the following result concerning the ideal of Definition~\ref{defn53}:

\begin{fact}[Eisworth, \cite{MR2506189}]\label{fact54} Suppose that $\lambda$ is a singular cardinal,
$\vec C=\langle C_\delta\mid \delta\in E^{\lambda^+}_{\cf(\lambda)}\rangle$ is a $C$-sequence such that $\otp(C_\delta)<\lambda$ for all $\delta\in E^{\lambda^+}_{\cf(\lambda)}$.
Whenever there are $\theta$-many pairwise disjoint $J(\vec C)$-positive sets,
$\lambda^+\nrightarrow[\lambda^+]^2_\theta$ holds.\footnote{Compare with \cite[Remark~2.10(2)]{Sh:413}.}
\end{fact}

\begin{cor} Suppose that $\lambda$ is a singular cardinal,
and $\vec C=\langle C_\delta\mid \delta \in E^{\lambda^+}_{\cf(\lambda)}\rangle$
is a club guessing $C$-sequence such that $\otp(C_\delta)=\lambda$ for all $\delta\in E^{\lambda^+}_{\cf(\lambda)}$.
Then $\lambda^+ \nrightarrow [\lambda^+]^2_{\lambda}$ holds.
\end{cor}
\begin{proof}
Using a general partition theorem for club guessing \cite[\S4.4]{paper46},
we may fix a partition $\langle S_\tau\mid \tau<\lambda\rangle$ of $E^{\lambda^+}_{\cf(\lambda)}$ such that $\vec C\restriction S_\tau$ guesses clubs for each $\tau < \lambda$.
By Corollary~\ref{cor24}(1), we may assume that $\lambda=\aleph_\lambda$,
hence, we may find a pairwise  disjoint sequence $\vec K=\langle K_\tau\mid \tau<\lambda\rangle$
of cofinal subsets of $\reg(\lambda)$ of order-type $\cf(\lambda)$.
For every $\tau<\lambda$,  set $B_\tau:=\{ \beta<\lambda^+\mid \cf(\beta)\in K_\tau\}$,
and note that $\vec B=\langle B_\tau\mid \tau<\lambda\rangle$ is a sequence of pairwise disjoint stationary subsets of $\lambda^+$.

Next, for all $\tau<\lambda$ and $\delta\in S_\tau$, let $D_\delta$ be the closure below $\delta$ of the following set $$\{\gamma\in C_\delta\mid \otp(C_\delta\cap\gamma)\in K_\tau\},$$
so that $D_\delta$ is a subclub of $C_\delta$ of order-type $\cf(\lambda)$. In particular, $\vec D:=\langle D_\delta\mid\delta\in E^{\lambda^+}_{\cf(\lambda)}\rangle$
is an uninhibited club-guessing $C$-sequence such that $\otp(D_\delta)<\lambda$ for all $\delta\in E^{\lambda^+}_{\cf(\lambda)}$.
A moment's reflection makes it clear that, for every $\tau<\lambda$, $B_\tau$ is $J(\vec D\restriction S_\tau)$-positive.
In particular, $\vec B$ consists of $\lambda$-many pairwise disjoint $J(\vec D)$-positive sets.
By Fact~\ref{fact54}, then, $\lambda^+\nrightarrow[\lambda^+]^2_\lambda$ holds.
\end{proof}

The preceding motivates the following stronger form of Problem~\ref{p10}:

\begin{q}\label{p14} Suppose that $\lambda$ is a singular cardinal. Is there a club guessing $C$-sequence $\vec C=\langle C_\delta\mid \delta\in E^{\lambda^+}_{\cf(\lambda)}\rangle$  such that $\otp(C_\delta)=\lambda$  for all $\delta\in E^{\lambda^+}_{\cf(\lambda)}$?
\end{q}

To compare, a negative answer is known for $\lambda$ regular,
as Abraham and Shelah proved \cite{AbSh:182}
it is consistent  for a regular cardinal $\lambda$ to have $\lambda^{++}$ many  clubs in $\lambda^+$ such that the intersection of any $\lambda^+$ many of them has size $<\lambda$.

Problem~\ref{p14} bears similarity to questions concerning the failure of diamond at successor of singulars;
results in this direction may be found in \cite{Sh:186,Sh:667,paper08}.

An affirmative answer to Problem~\ref{p14} (or Problem~\ref{p10}) under the additional assumption of $\square^*_\lambda$ is of interest, as well.

We conclude this paper by reiterating Question~2 from \cite{paper11}:

\begin{q} Suppose that $\lambda$ is a singular cardinal. Assuming $\square_\lambda$ holds,
may it be witnessed by $C$-sequence $\vec C=\langle C_\delta\mid\delta<\lambda^+\rangle$
with the property that for every club $D\s\lambda^+$, there exists some $\delta\in\acc(\lambda^+)$
such that $\otp(C_\delta\cap D)=\lambda$?
\end{q}

By the main result of \cite{paper19}, an affirmative answer to the preceding holds assuming $2^\lambda=\lambda^+$.

\section{Acknowledgments}
This survey is an expanded version of a talk given by the author at the ``Perspectives on Set Theory'' conference, November 2023.
We thank the organizers for the invitation to this unusual conference.

The author was supported by the European Research Council (grant agreement ERC-2018-StG 802756)
and by the Israel Science Foundation (grant agreement 665/20).

\newcommand{\etalchar}[1]{$^{#1}$}

\end{document}